\RequirePackage[english]{babel}
\documentclass[12pt,twoside,reqno]{amsart}

\usepackage[OT1]{fontenc}
\usepackage{type1cm}

\usepackage{enumerate}

\usepackage{amsthm}
\RequirePackage{amsmath,amsfonts,amssymb,amsthm}

\RequirePackage[dvips]{graphicx}

\usepackage{psfrag}
\usepackage{verbatim}
\usepackage[usenames]{color}
\usepackage[dvips]{graphicx} % for gfx

  \numberwithin{equation}{section}

\usepackage[active]{srcltx}  % for the inverse search

  \newcommand{\N}{\mathbb{N}}         % natural numbers
  \newcommand{\T}{\mathbb{T}}		  % torus

  \newcommand{\diam}{\text{diam}}       % diameter
  \newcommand{\inrad}{\text{inrad}}
  \newcommand{\e}{\varepsilon}

  \newtheorem{theorem}{Theorem}[section]
  \newtheorem{lemma}[theorem]{Lemma}
  \newtheorem{prop}[theorem]{Proposition}
  \newtheorem{cor}[theorem]{Corollary}

  \theoremstyle{remark}
  \newtheorem{rem}[theorem]{Remark}

\DeclareSymbolFont{bbold}{U}{bbold}{m}{n}
\DeclareSymbolFontAlphabet{\mathbbold}{bbold}

\begin{document}

\title[A note on the hitting probabilities]{A note on the hitting probabilities of random covering sets}

\author{Bing Li$^{1,2}$}

\author{Ville Suomala$^{2}$}
%\address{Department of Mathematical Sciences \\
%         P.O. Box 3000 \\
%         FI-90014 University of Oulu \\
%         Finland}
%\email{ville.suomala@oulu.fi}
\address{$^1$Department of Mathematics, South China University of Technology,
Guangzhou, 510641, P.R. China}
\address{$^2$Department of Mathematical Sciences, P.O. Box 3000, 90014
University of Oulu, Finland}

\email{libing0826@gmail.com, ville.suomala@oulu.fi}

\begin{abstract}
Let $E=\limsup\limits_{n\to\infty}(g_n+\xi_n)$ be the random covering set on the torus $\mathbb{T}^d$, where $\{g_n\}$ is a sequence of ball-like sets and $\xi_n$ is a sequence of independent random variables uniformly distributed on $\T^d$. We prove that $E\cap F\neq\emptyset$ almost surely whenever $F\subset\mathbb{T}^d$ is an analytic set with Hausdorff dimension, $\dim_H(F)>d-\alpha$, where $\alpha$ is the almost sure Hausdorff dimension of $E$. Moreover, examples are given to show that the condition on $\dim_H(F)$ cannot be replaced by the packing dimension of $F$.
\end{abstract}

\maketitle



\section{Introduction}

Let $(g_n)$ be a sequence of subsets of the $d$-dimensional torus $\T^d$ and $(\xi_n)$ a sequence of independent and uniformly distributed random variables on $\T^d$. Let $(\Omega, \mathbb{P})$ be the corresponding probability space and consider the random translates $G_n=g_n+\xi_n$. We are interested in the random covering set
\[E=\limsup_{n\to\infty}G_n=\bigcap_{n=1}^\infty\bigcup_{k=n}^\infty G_k\,,\]
that is, the set of points in $\mathbb{T}^d$ covered infinitely often by $(G_n)$.
Applying the Borel-Cantelli lemma and Fubini's theorem, the Lebesgue measure of $E$, $\mathcal{L}(E)$, is almost surely zero or one according to the convergence or divergence of $\sum_{n=1}^\infty\mathcal{L}(g_n)$ (see \cite{Ka1}).

The random covering problem on the circle $\mathbb{T}:=\mathbb{T}^1$ ($d=1$), where $g_n$ are intervals on the circle with length $l_n$, has been extensively studied in the literature. When $\sum_{n=1}^\infty l_n<\infty$, that is, $\mathcal{L}(E)=0$, Durand \cite{Durand} (see also \cite{FanWu}) showed that the almost sure Hausdorff dimension of the covering set is
\begin{equation*}\label{Def:alpha1}
\dim_H(E)=\sup\{0\le s\le 1:
\sum_{n=1}^{\infty}l_n^s=\infty\}:=\alpha\,.
\end{equation*}
 Under the following extra condition (C),
\begin{itemize}
\item[(C)]\, There exists an increasing sequence of positive
integers $\{k_i\}$ such that
\begin{equation*}\label{Eq:k}
\lim_{i \to \infty} \frac{k_{i+1}} {k_i} = 1\ \ \text{and}\ \
\lim_{i \to\infty}\frac{\log_2 n_{k_i}}{k_i}= \alpha<1,
\end{equation*}
where $
n_k=\#\{n\in\mathbb{N} :l_n\in [2^{-k+1},2^{-k+2})\}
\qquad (k\geq 2),
$
\end{itemize}
 Li, Shieh and Xiao \cite{LSX} (see also \cite{KPX}) proved that the probability of $E$ hitting a deterministic analytic set $F\subset\mathbb{T}$,
 \begin {equation}\label{hitting}
\mathbb{P}\big(E\cap F\neq\emptyset\big)=
\begin {cases} 0  &\text{if}\ \  \dim_P(F)<1-\alpha,\\
1 &\text{if}\ \ \dim_P(F)>1-\alpha,
\end {cases}
\end {equation}
where $\dim_P(F)$ is the packing dimension of $F$. 
Moreover, they obtained estimaties on the Hausdorff dimension of the intersection $E\cap F$,
\begin{equation}\label{dimHestimation}
\dim_H (F)+\alpha-1\leq \dim_H(E\cap F)\leq \dim_P(F)+\alpha-1\ \
\ a.s.
\end{equation}
In fact, the probability zero part of \eqref{hitting} and the first inequality  of \eqref{dimHestimation} remain valid even without the extra condition (C) as it is not used in the corresponding proofs in \cite{LSX}.
We mention that the proofs of \eqref{hitting} and \eqref{dimHestimation} in \cite{LSX} can be easily adapted to the higher dimensional torus $\mathbb{T}^d$ when $g_n$ are balls in $\mathbb{T}^d$. It was left open whether the  probability one part of \eqref{hitting} holds without the assumption (C) and the main purpose of this note is to settle this question.  

Now we return to the $d$-dimensional case. Let $l_n=\diam(g_n)$. For simplicity, we assume that all the $g_n$ are balls; $g_n=B(0,l_n/2)$. All results of this paper (with trivial modifications in the proofs) hold for sets $g_n$ which are ball-like in the sense that
\[\lim_{n\to\infty}\frac{\log\inrad(g_n)}{\log l_n}=1\, ,\]
where $\inrad(g_n)$ denotes the maximal radius of the balls inside $g_n$.
By reordering, we can assume that $(l_n)$ is decreasing.
It is well known (see \cite{FanWu},\cite{Durand},\cite{JJKLS},\cite{Per}) that the almost sure Hausdorff dimension of $E$ is given by the formulae
\begin{equation}\label{dimE}
\dim_H(E)=\alpha(l_n):=\limsup_{n\to\infty}\frac{\log n}{-\log l_n}=\sup\{0\le s\le d\,:\sum_{n=1}^\infty l_{n}^s=\infty\}.
\end{equation}

Our main result is the following theorem concerning the probability one part of \eqref{hitting}.  Here the extra condition (C) for $\{l_n\}$ is relaxed and the condition on $\dim_P(F)$ in \eqref{hitting} is replaced by $\dim_H(F)$.

\begin{theorem}\label{main_result}
If $F\subset\mathbb{T}^d$ is an analytic set with $\dim_H (F)>d-\alpha$, then
$E\cap F\neq\emptyset$ almost surely.
\end{theorem}
Combining Theorem \ref{main_result} and the probability zero part of \eqref{hitting}, we have the following hitting probability result, which applies also in case the condition (C) fails.
\begin{cor}\label{resultwithoutC}
Let $F\subset \mathbb{T}^d$ be an analytic set. Then
 \begin {equation}\label{hittingwithoutcondition}
\mathbb{P}\big(E\cap F\neq\emptyset\big)=
\begin {cases} 0  &\text{if}\ \  \dim_P(F)<d-\alpha,\\
1 &\text{if}\ \ \dim_H(F)>d-\alpha.
\end {cases}
\end {equation}
Furthermore,
\begin{equation}\label{dimHestimationwithoutcondition}
\dim_H (F)+\alpha-d\leq \dim_H(E\cap F)\leq \dim_P(F)+\alpha-d\ \
\ a.s.
\end{equation}
\end{cor}
We give examples indicating that in general, Theorem \ref{main_result} does not hold if $\dim_H (F)$ is replaced by $\dim_P (F)$, thus showing the necessity of the extra assumption (C) in \cite{LSX}.

\begin{prop}\label{prop:cex}
There are $(l_n)$ such that $\alpha(l_n)=d$ and a closed set $F\subset\T^d$ with $\dim_P (F)=d$ while $E\cap F=\emptyset$ almost surely.
\end{prop}

\begin{prop}\label{prop:dimcex}
For all $0\le \alpha,t\le d$, there are a sequence $(l_n)$ with $\alpha=\alpha(l_n)$ and a closed set $F$ with $\dim_H F=t, \dim_P(F)=d$ such that almost surely, $\dim_H(E\cap F)=\min\{\alpha,t\}$. In particular, it is possible that
a.s. $\dim_H(F)+\alpha-d <\dim_H(E\cap F)<\dim_P(F)+\alpha-d$.
\end{prop}
\begin{rem}
Proposition \ref{prop:dimcex} shows that both of the inequalities in \eqref{dimHestimationwithoutcondition} can be strict. Meanwhile, for any $\{g_n\}$ with $0<\alpha<d$, Proposition \ref{prop:dimcex} also gives an example of $F$ satisfying $\dim_H(F)<d-\alpha$, but $\dim_H(E\cap F)=\alpha>0$ a.s., in particular, $\mathbb{P}(E\cap F\neq\emptyset)=1$, which means that probability zero part of \eqref{hittingwithoutcondition} does not hold if $\dim_P(F)$ is replaced by $\dim_H(F)$ in Corollary \ref{resultwithoutC}. 

As indicated in \cite{LSX}, the hitting probabilities of the random covering sets are closely related to the hitting probabilities of certain limsup random fractals considered e.g. in \cite{KPX}. Although we don't make it explicit, it follows from the examples in Proposition \ref{prop:cex} and \ref{prop:dimcex} that an assumption analogous to (C), called \emph{the index assumption} (Condition 4) in \cite{KPX}, is essential for the validity of the results of \cite{KPX}.

Although the used methods are somewhat different, there is a close conceptual connection between the hitting probability estimates of random sets and the intersection estimates of $F$ and $f(G)$, where $F, G\subset\mathbb{R}^d$ are deterministic sets and $f$ is a 'typical' element of a suitable family of transformations $\mathbb{R}^d\rightarrow\mathbb{R}^d$. We refer to \cite[\S 13]{Mat} for an overview of such results.
\end{rem}

\section{Proofs}
For $N\in\N$, we use the notation $[N]=\{1,\ldots, N\}$.
Let $\mathcal{Q}_n$ denote the level $n$ dyadic grid of $\T^d$. For each $n$, we may label the elements of $\mathcal{Q}_n$ as $\{Q_1,\ldots,Q_{2^{nd}}\}$. We say that $Q\in\mathcal{Q}_n$ is uniformly distributed, if $Q=Q_X$, where $X$ is a random variable with $\mathbb{P}(X=i)=2^{-nd}$ for each $i\in[2^{nd}]$. We use  similar terminology as well when $\mathcal{Q}_n$ is replaced by some subfamily, e.g. all the elements of $\mathcal{Q}_n$ that lay inside a given cube $Q\in\mathcal{Q}_m$, $m\le n$. We denote such a family by $\mathcal{Q}_n(Q)$.

To avoid boundary effects, we assume throughout the proof of Theorem \ref{main_result} and the preceeding lemmata, that for each $n$, $\cup_{Q\in\mathcal{Q}_n}Q$ is a disjoint cover and we consider on $\mathbb{T}^d$ the topology induced by the dyadic cubes $Q\in\mathcal{Q}_n$, $n\in\N$. This is not a restriction of generality, since it is well known that e.g. the half-open dyadic cubes induce the standard Borel sigma algebra on $\mathbb{T}^d$, and hence the same analytic sets as the Euclidean topology.

Theorem \ref{main_result} is obtained as a consequence of several lemmata.

\begin{lemma}\label{lem:uniform_dim_bound}
If $F\subset\T^d$ is an analytic set and $\dim_H F>s$, then there is a nonempty closed subset $H\subset F$ such that $\dim_H(Q\cap H)>s$ for all dyadic cubes $Q$ for which $Q\cap H\neq\emptyset$.
\end{lemma}

\begin{proof}
First, we may find a closed set $K\subset F$ with $\mathcal{H}^t(K)>0$ for some $t>s$ (see \cite[Corollary 2,p. 99]{Ro70}). Let $U=\{x\in K\,:\,\dim_H(Q\cap K)\le s\text{ for some dyadic cube }Q\ni x\}$. Then $U$ is relatively open in $K$ and whence $H=K\setminus U$ is closed. It is clear that $\dim_H(Q\cap H)>s$ whenever $Q$ is a dyadic cube touching $H$. Moreover, a simple covering argument implies that $\mathcal{H}^t(U)=0$, whence $\mathcal{H}^t(H)=\mathcal{H}^t(K)>0$, and in particular $K$ is nonempty.
\end{proof}

The following lemma is a direct consequence of the definition of the Hausdorff measure.

\begin{lemma}\label{lem:hdim_lowerbound}
If $Q\in\mathcal{Q}_{n_0}$ is a dyadic cube, $F\subset\T^d$ and $\dim_H (F\cap Q)>s$, then there is $n_0\le N\in\N$ such that for $n\ge N$, there are at least $2^{ns}$ subcubes of $Q$ in $\mathcal{Q}_n$ which touch $F$.
\end{lemma}

\begin{lemma}\label{lem:existence_random_subcubes}
Suppose that $\alpha=\alpha(l_n)>t$ and let $Q\in\mathcal{Q}_{n_0}$ be given. For $n\ge n_0$, and each $l_j$ with $2^{-n}\sqrt{d}\le l_j\le 2^{-n_0}\sqrt{d}$, let $Q^j\in\mathcal{Q}_n$ be the dyadic cube containing $\xi_j$. Consider the random variable $N(Q,n)=\#\{j\,:\,Q^j\subset Q\}$. Then,
\begin{equation}\label{eq:manyhits}
\limsup_{n\rightarrow \infty}\mathbb{P}\left(N(Q,n)\ge 2^{nt}\right)=1\,.
\end{equation}
\end{lemma}

\begin{proof}
Pick $\alpha>r>t$. From the definition of $\alpha$, it readily follows that there are arbitrarily large $n$ such that the number of indices $j$ with
$2^{-n}\sqrt{d}\le l_j\le 2^{-n_0}\sqrt{d}$ is $L_n\ge 2^{nr}$. For each of these $j$,
 $Q^j$ is uniformly distributed among $\mathcal{Q}_n$, and clearly, $\{Q^j\}_{2^{-n}\sqrt{d}\le l_j\le 2^{-n_0}\sqrt{d}}$ are mutually independent random variables.
Write $X_j$ for the indicator function of $\{Q^j\subset Q\}$. Then $\mathbb{E}(X_j)=2^{-n_0d}$. Thus
\begin{align*}
\mathbb{E}(N(Q,n))&=\sum_j\mathbb{E}(X_j)=2^{-n_0d}L_n\,,\\
\mathbb{E}(N(Q,n)^2)&=\sum_j\mathbb{E}(X_j)+\sum_{j\neq i}\mathbb{E}(X_j X_i)=2^{-n_0d}L_n+(L_n^{2}-L_n)2^{-2n_0d}\,.
\end{align*}
Applying Chebyshev's inequality, $\mathbb{P}\left(N(Q,n)<\frac12\mathbb{E}(N(Q,n))\right)$ is bounded from above by
\begin{eqnarray*}
\mathbb{P}\left(|N(Q,n)-\mathbb{E}(N(Q,n))|\geq \frac12\mathbb{E}(N(Q,n))\right)\le 2^{n_0d+2}L_n^{-1}\le 2^{-nr+n_0d+2}\,.
\end{eqnarray*}
As $2^{nt}\le 2^{nr-n_0d-1}\le L_n 2^{-n_0d-1}=\frac12\mathbb{E}(N(Q,n))$ for arbitrarily large values of $n$, the claim follows.
\end{proof}

\begin{rem}\label{rem:uniformsequence}
It is clear from the above proof that the sequence realising the limsup in \eqref{eq:manyhits} can be chosen to be independent of the cube $Q$ as it only depends on the sequence $(l_n)_n$. More precisely, there is a sequence $n_k\to\infty$ such that for each dyadic cube $Q$,
\[
\lim_{k\rightarrow \infty}\mathbb{P}\left(N(Q,n_k)\ge 2^{n_kt}\right)=1\,.
\]
\end{rem}

\begin{lemma}\label{lem:cubes_coincide}
Let $Q\in\mathcal{Q}_{n_0}$ be a dyadic cube and let $n_0\le n\in\N$. Suppose that $Q_i$, $i\in[K]$, $K\ge 2^{ns}$ are (deterministic) cubes in $\mathcal{Q}_n(Q)$ and let $Q^{j}$, $j\in[L]$, $L\ge 2^{nt}$ be uniformly distributed independent random cubes in $\mathcal{Q}_n(Q)$. Then
\begin{equation}
\mathbb{P}\left(Q^{j}=Q_i\text{ for some }i\in[K], j\in[L]\right)\ge 1-\e(s,t,n)
\end{equation}
where $\e(s,t,n)\to 0$ as $n\to\infty$, provided $s+t>d$.
\end{lemma}

\begin{proof}
For each $Q^j$, we have
\[\mathbb{P}(Q^j\neq Q_i\text{ for all }i)= 1-K2^{(n_0-n)d}\le1-2^{n_0}2^{n(s-d)}\,.\]
Hence, by independence,
\begin{align*}
&\mathbb{P}\left(Q^j\neq Q_i\text{ for all }i\text{ and }j\right)\le(1-2^{n_0}2^{n(s-d)})^{L}\le\left(1-2^{n_0}2^{n(s-d)}\right)^{2^{nt}}
\end{align*}
and this upper bound tends to zero as $n\rightarrow+\infty$, since $s+t>d$.
\end{proof}

\begin{proof}[Proof of Theorem \ref{main_result}]
By Lemma \ref{lem:uniform_dim_bound}, we may assume that $F$ is closed and that $\dim_H (F\cap Q)>s>d-t>d-\alpha$ for some $s,t$ and for all dyadic cubes $Q$ intersecting $F$. Fix $0<\varepsilon_k<1$ such that $\sum_{k}\varepsilon_k<\infty$. 

The following notation is adapted from Lemma \ref{lem:existence_random_subcubes}. Given $m,n\in\N$, $n\ge m$ and $Q\in\mathcal{Q}_m$ we consider those $\xi_j\in Q$ for which $2^{-n}\sqrt{d}\le l_j\le 2^{-m}\sqrt{d}$, and let $Q^j\in\mathcal{Q}_n(Q)$ be the dyadic cube containing $\xi_j$. After re-enumeration, we denote by $\{Q^j\}_{j=1}^{N(Q,m)}$ the random family of all such cubes.

We define a sequence of integers $(n_k)_k$ in the following manner. To begin with, we choose $n_1$ so large that
\begin{enumerate}
\item There are at least $2^{n_1s}$ subcubes in $\mathcal{Q}_{n_1}$ intersecting $F$.
\item\label{2} 
$N(\mathbb{T}^d,n_1)\ge 2^{n_1 t}$.
\item The probability that at least one cube in  $\{Q^j\}_{j=1}^{N(\mathbb{T}^d,n_{1})}$ 
intersects $F$ is at least $1-\varepsilon_1$.
\end{enumerate}
We observe that such a choice is possible by Lemmata \ref{lem:hdim_lowerbound}--\ref{lem:cubes_coincide} (In fact, Lemma \ref{lem:existence_random_subcubes} is not even needed for the choice of $n_1$).

For $k\in\N$, we define $n_{k+1}$ inductively such that for each $Q\in\mathcal{Q}_{n_k}$ intersecting $F$, the following conditions hold:
\begin{enumerate}
\item\label{1} There are at least $2^{n_{k+1}s}$ cubes in $\mathcal{Q}_{n_{k+1}}(Q)$ intersecting $F$.
\item With probability at least $1-\varepsilon_{k+1}$, $N(Q,n_{k+1})\ge 2^{n_{k+1} t}$.
\item\label{3} 
Conditioned on $N(Q,n_{k+1})\ge 2^{n_{k+1} t}$, the probability that at least one cube in $\{Q^j\}_{j=1}^{N(Q,n_{k+1})}$ intersects $F$ is at least $1-\varepsilon_{k+1}$.
\end{enumerate}
Again, such choices are possible by Lemmata \ref{lem:hdim_lowerbound}--\ref{lem:cubes_coincide} since there are only finitely many such $Q\in\mathcal{Q}_{n_k}$ (For \eqref{2} also take Remark \ref{rem:uniformsequence} into account).

Let $\mathcal{A}_k$ denote the event that there are $Q^1,\ldots, Q^k$ satisfying for all $i\in[k]$ the conditions,
\begin{itemize}
\item $Q^i\in\mathcal{Q}_{n_i}$,
\item $Q^i\cap F\neq\emptyset$,
\item $Q^{i+1}\subset Q^i$,
\item  There is $\xi_j\in Q^i$ with $2^{-n_i}\sqrt{d}\le l_j\le 2^{-n_{i-1}}\sqrt{d}$ (and consequently $Q^i\subset G_{j}$).
\end{itemize}

These are decreasing events, and it follows from the above conditions \eqref{1}-\eqref{3} that
\[\mathbb{P}\left(\mathcal{A}_n\,|\,\mathcal{A}_{n-1}\right)\ge 1-2\varepsilon_{n}\,,\]
Since $\sum_{k}\varepsilon_k<\infty$, this yields $\mathbb{P}(\cap_n\mathcal{A}_n)>0$.
Clearly
\[\{F\cap E\neq\emptyset\}\supset\cap_n\mathcal{A}_n\,,\]
and consequently $\mathbb{P}(F\cap E\neq\emptyset)>0$. Finally, $\{F\cap E\neq\emptyset\}$ is obviously a tail event and the claim follows from the Kolmogorov zero-one law.
\end{proof}

\begin{proof}[Proof of Corollary \ref{resultwithoutC}]
The equalities in \eqref{hittingwithoutcondition} hold by Theorem \ref{main_result} and the probability zero part of \eqref{hitting}.

The right-hand inequality of \eqref{dimHestimationwithoutcondition} can be obtained by the same proof as the corresponding part of \eqref{dimHestimation} in \cite{LSX}.

The left-hand inequality of \eqref{dimHestimationwithoutcondition} follows from Theorem \ref{main_result} and Lemma 3.4 in \cite{KPX}.
\end{proof}

\begin{proof}[Proof of Proposition \ref{prop:cex}]
We present the construction for $d=1$. The generalisation for $d>1$ is straightforward.

Let $\varepsilon_k>0$ be such that $\sum_k\varepsilon_k<+\infty$ and let $0<s_k<1$ be increasing to 1 as $k\to\infty$. Also, let $m_k< n_k$ be two increasing sequences of integers to be determined later. We construct the set $F\subset[0,1]$ as follows. First, we divide $[0,1]$ into $2^{m_1}$ intervals of length $2^{-m_1}$ and inside each of these, we select an interval of length $2^{-n_1}$. Let $\mathcal{I}_1$ denote the collection of all these selected intervals (called the first level construction intervals).

We continue inductively. Assuming that $\mathcal{I}_k$ is a family of $N_k=\prod_{i=1}^k 2^{m_i}$ disjoint intervals of length $\delta_k=\prod_{i=1}^k2^{-n_i}$, we decompose each element of $\mathcal{I}_k$ into disjoint subintervals of length $2^{-m_{k+1}}\delta_k$ and inside each of these, select one interval of length $\delta_{k+1}:=2^{-n_{k+1}}\delta_k$. We denote these $N_{k+1}=2^{m_{k+1}}N_k$ intervals of length $\delta_{k+1}$ by $\mathcal{I}_{k+1}$.  
Let $F_k=\cup_{I\in\mathcal{I}_k} I$ and $F=\cap_{k}F_k$.

We choose each $m_k$ so large (depending on the choices of $m_i$, $n_i$ for $i<k$) that $2^{m_k} N_{k-1}(2^{-m_k}\delta_{k-1})^{s_k}\ge1$. This readily implies that $\dim_P(F)\ge \limsup\limits_{k\to\infty}s_k= 1$ (see \cite{FWW97}). Thus $\dim_P(F)=1$. 

To obtain suitable random covering sets, we set
\begin{equation*}
g_n=[0,\delta_k]\text{ for }2^{n_{k-1}s_{k-1}}\le n<2^{n_k s_k}\,.
\end{equation*}
and denote
\[E_k=\bigcup_{2^{n_{k-1}s_{k-1}}\le n<2^{n_k s_k}}[\xi_n-\frac{\delta_k}{2},\xi_n+\frac{\delta_k}{2}]\,,\]
where $\xi_n$ are independent and uniformly distributed on $\mathbb{T}$.
That is, $l_n=\delta_k$ for $2^{n_{k-1}s_{k-1}}\le n<2^{n_k s_k}$. It is clear that such $(l_n)$ does not satisfy the condition (C).
It follows that $\alpha(l_n)=1$ provided $n_k$ grows sufficiently fast. On the other hand, we have the estimate
\[\mathbb{P}\left(E_k\cap F_k\neq\emptyset\right)\le  3 N_k2^{n_k s_k}\delta_k=3 N_k  2^{n_k s_k}2^{-n_k}\delta_{k-1}\,,\]
and this can be made smaller than $\varepsilon_k$ by choosing $n_{k}$ large enough, depending on $s_k$ and the previous choices of $n_i$ and $m_j$ for $i<k$, $j\le k$.

The events $\{E_k\cap F_k\neq\emptyset\}$ are independent for different values of $k$ and thus the Borel-Cantelli lemma implies that almost surely, $E_k\cap F_k=\emptyset$ when $k$ is large. Since $E=\limsup\limits_{k\rightarrow\infty} E_k$ and $F\subset F_k$ for each $k$, this yields $E\cap F=\emptyset$ almost surely.
\end{proof}

For the proof of Proposition \ref{prop:dimcex}, we require the following elementary covering estimate.

\begin{lemma}\label{lemma:Dvoreasy}
Let $0<\beta<\alpha<1$ and $0<c\le C<+\infty$. Suppose $\xi_n$ are independent and uniformly distributed random variables on $\mathbb{T}$. Then
\[\mathbb{P}\left(\mathbb{T}=\bigcup_{C<n\le c\eta^{-\alpha}}[\xi_n-\frac{\eta^\beta}{2},\xi_n+\frac{\eta^\beta}{2}]\right)\longrightarrow 1\,,\]
as $\eta\downarrow 0$. 
\end{lemma}

\begin{proof}
We may cover $\mathbb{T}$ with less than $3/\eta^\beta$ intervals $I$ of length $\eta^\beta/2$. For each of these $I$ and each $n\ge C$, we have
$\mathbb{P}\left(I\subset[\xi_n-\frac{\eta^\beta}{2},\xi_n+\frac{\eta^\beta}{2}]\right)= \eta^\beta/2$ and since $\{I\subset[\xi_n-\frac{\eta^\beta}{2},\xi_n+\frac{\eta^\beta}{2}]\}$ are independent events,
\[\mathbb{P}\left(I\not\subset\bigcup_{C\le n\le c \eta^{-\alpha}}[\xi_n-\frac{\eta^\beta}{2},\xi_n+\frac{\eta^\beta}{2}]\right)\le (1-\eta^\beta/2)^{c\eta^{-\alpha}-C}\,.\]
Summing over all $I$ yields
\begin{align*}
&\mathbb{P}\left(\mathbb{T}\neq\bigcup_{C\le n\le\eta^{-\alpha}}[\xi_n-\frac{\eta^\beta}{2},\xi_n+\frac{\eta^\beta}{2}]\right)
\le 3\eta^{-\beta}\left(1-\eta^\beta/2\right)^{c\eta^{-\alpha}-C}\longrightarrow 0\,,
\end{align*}
as $\eta\downarrow 0$.
\end{proof}

\begin{proof}[Proof of Proposition \ref{prop:dimcex}]
For simplicity, we again assume that $d=1$.

Let $n_1< m_1<n_2<m_2<n_3<\cdots$ be increasing sequences of integers (to be determined later). Denote by $\lfloor x\rfloor$ the integer part of $x$.  We construct $F$ by an inductive process as follows. We first decompose $[0,1]$ into $N_1=2^{\lfloor n_1 t\rfloor}$ intervals of length $2^{-\lfloor n_1 t\rfloor}$ and further choose one sub-interval of length $2^{-n_1}$ inside each. These form the family $\mathcal{I}_1$.

The construction is continued inductively. Given $\mathcal{I}_k$, a family of $N_k$ disjoint intervals of length $\delta_k=2^{-n_1-\ldots-n_k}$. We decompose each element of $\mathcal{I}_k$ into subintervals of length $2^{-\lfloor n_{k+1}t\rfloor}\delta_k$ and choose one subinterval of length $\delta_{k+1}=2^{-n_{k+1}}\delta_k$ inside each.
In total, there will be $N_{k+1}=\lfloor 2^{n_{k+1} t}\rfloor N_k$ such intervals with length $\delta_{k+1}$, and these form the family $\mathcal{I}_{k+1}$. We let $F_0=[0,1]$, $F_k=\cup_{I\in\mathcal{I}_k} I$ and finally $F=\cap_k F_k$. It is straightforward to check that $\dim_H F=t$ (see \cite{FWW97}).

To define the random covering sets, we denote $\eta_k=2^{-m_1-\ldots-m_k}$, let
\begin{equation*}
g_n=[0,\eta_k]\text{ for }2^{m_{k-1}\alpha}\le n<2^{m_k \alpha}\,,
\end{equation*}
and denote
\[E_k=\bigcup_{2^{m_{k-1}\alpha}\le n<2^{m_k \alpha}}[\xi_n-\frac{\eta_k}{2},\xi_n+\frac{\eta_k}{2}]\,,\]
where again $\xi_n$ are independent and uniformly distributed on $\mathbb{T}$.
Choosing $m_k$ large enough, we can check from \eqref{dimE} that the Hausdorff dimension of $E$ is $\alpha(l_n)=\alpha$ almost surely.

Obviously, $\dim_H(E\cap F)\le\min\{\dim_H F,\dim_H E\}$, so it remains to show that it is possible to choose the parameters $n_k, m_k$ such that also
\begin{equation}\label{eq:dim_lowerbound}
\dim_H(E\cap F)\ge\min\{t,\alpha\}
\end{equation}
holds almost surely. The reason why this should be true is that while the Hausdorff dimension of $F$ is realised on scales $\delta_k$, the Hausdorff dimension of $E$ is realised on scales $\eta_k$, $\delta_k\gg\eta_k\gg\delta_{k+1}$.
On scales $\delta_k$, $E$ is rather uniformly distributed (with high probability) and correspondingly, $F$ looks "one dimensional" on the scales $\eta_k$. So in order to find an efficient covering for $E\cap F$ one has to use intervals of size $\eta_k$ (roughly $\eta_k^{-\alpha}$ are needed) or $\delta_{k}$ (roughly $\delta_k^{-t}$ are needed), but since these scales are not comparable, one essentially does not gain anything by looking at the covering formed by intersecting the elements of these two 'natural' coverings. For deterministic sets with same kind of intersection behaviour, see e.g. \cite[Example 13.19]{Mat}.

Now to the detailed proof of \eqref{eq:dim_lowerbound}.
We would like to use the general mass transference principle of Beresnevich and Velani \cite[Theorem 3]{BeresVelani} since it is often very handy in this kind of situations. However, there is a monotonicity assumption for the ratio of the gauge function in their result, which cannot be verified in the situation at hand. Fortunately, our construction of the set $F$ and the random sets $E$ is regular enough, so that we can still use the main idea from their proof.

To that end, we construct a Cantor type set $G$ inside $E\cap F$ with the help of Lemma \ref{lemma:Dvoreasy}. Pick an increasing sequence $(\beta_k)$ with $\lim\limits_{k\to\infty}\beta_k=\alpha$ and let $\varepsilon_k>0$ such that $\sum_{k}\varepsilon_k<+\infty$. Then, by choosing each $m_k$ large enough compared to $m_{k-1}$, Lemma \ref{lemma:Dvoreasy} guarantees that with probability at least $1-\varepsilon_k$, we have 
\begin{equation}\label{full_cover}
\bigcup_{2^{m_{k-1}\alpha}\le n<2^{m_k \alpha}}[\xi_n-\frac{\eta^{\beta_k}_k}{2},\xi_n+\frac{\eta^{\beta_k}_k}{2}]=\mathbb{T}\,.
\end{equation}
Since the events \eqref{full_cover} are independent for disjoint values of $k$, the Borel-Cantelli lemma implies that with positive probability, \eqref{full_cover} holds true for all $k$ simultaneously.

From now on, we pick such $\omega$ that \eqref{full_cover} is valid for all $k\in\mathbb{N}$. For each $k$, we define families $\widetilde{\mathcal{I}}_k, \mathcal{G}_k$ such that $\widetilde{\mathcal{I}}_k\subset\mathcal{I}_k$ and $\bigcup_{J\in\mathcal{G}_k}J\subset E_k$.
% \[\mathcal{G}_k\subset\{[\xi_n-\frac{\eta_k}{2},\xi_n+\frac{\eta_k}{2}]\,:%\,2^{m_{k-1}\alpha}\le n<2^{m_k \alpha}\}\,.\] 
We begin by setting $\widetilde{\mathcal{I}}_1=\mathcal{I}_1$ and  continue inductively as follows; Suppose $\widetilde{\mathcal{I}}_k$ has been defined with $L_k:=\#\widetilde{\mathcal{I}}_k$. Since \eqref{full_cover} holds, for each $I\in\widetilde{\mathcal{I}}_k$ we can choose a disjoint subfamily of $\{[\xi_n-\frac{\eta^{\beta_k}_k}{2},\xi_n+\frac{\eta^{\beta_k}_k}{2}]\subset I\}$ containing $\lfloor\delta_k/(3\eta^{\beta_k}_k)\rfloor$ intervals (We choose $m_k$ large enough to guarantee $\eta_k^{\beta_k}<\delta_k/12$). For each of these intervals, we choose the concentric interval $[\xi_n-\frac{\eta_k}{2},\xi_n+\frac{\eta_k}{2}]$ to the collection $\mathcal{G}_k$. Thus, in particular $\cup_{J\in\mathcal{G}_k}J\subset E_k$. As a result of the construction, there are $M_k:=L_k \lfloor\delta_k/(3\eta^{\beta_k}_k)\rfloor$ elements in $\mathcal{G}_k$. The family $\widetilde{\mathcal{I}}_{k+1}$ is obtained by selecting $\lfloor\eta_k 2^{-\lfloor n_{k+1} t\rfloor}\rfloor-2$ intervals in $\mathcal{I}_{k+1}$ inside each $J\in\mathcal{G}_k$. Then 
\[L_{k+1}=M_k \lfloor\eta_k 2^{-\lfloor n_{k+1} t\rfloor}\rfloor-2\,.\]
Let $G=\cap_{k=1}^\infty\cup_{I\in\tilde{I}_k}I=\cap_{k=1}^\infty\cup_{J\in\mathcal{G}_k}J.$. Thus $G\subset E\cap F$.

By choosing each $m_k$ large enough depending on $\delta_k$, and further $n_{k+1}$ large enough depending on $\eta_k$, we can make sure that
\begin{align}
\label{eka}\lim_{k\rightarrow\infty}\frac{\log M_k}{-\log\eta_k}&=\alpha\,,\\
\label{toka}\lim_{k\rightarrow\infty}\frac{\log L_k}{-\log\delta_k}&=t\,.
\end{align}
 Now it is straightforward to check that $\dim_H (G)=\min\{\alpha,t\}$. Indeed, defining a probability measure $\mu$ supported on $G$ such that $\mu(I)=L_{k}^{-1}$ for each $I\in \widetilde{\mathcal{I}}_k$ (and consequently also $\mu(J)=M_{k}^{-1}$ for all $J\in\mathcal{G}_k$), it follows using \eqref{eka}-\eqref{toka} and the fact that the subintervals of any $I\in\widetilde{\mathcal{I}}_k$ (resp. $J\in\mathcal{G}_k$) in $\mathcal{G}_k$ (resp. $\widetilde{\mathcal{I}}_{k+1}$) are essentially uniformly distributed, that 
 \begin{equation}\label{localdim}
 \liminf_{r\downarrow 0}\frac{\log(\mu(B(x,r)))}{\log r}=\min\{\alpha,t\}
 \end{equation} 
 for all $x\in G$. Whence $\dim_H(E\cap F)\ge\dim_H(G)\ge\min\{\alpha,t\}$. We omit the detailed proof of \eqref{localdim} since this kind of results are well known. See e.g.\cite[Lemma 2.2]{FWW97} and observe that our Cantor set $G
 $ is essentially a homogeneous Cantor set in the notation of \cite{FWW97}.

We have now shown that $\dim_H(E\cap F)\ge\min\{\alpha,t\}$ with positive probability. Finally, $\dim_H(E\cap F)\ge\min\{\alpha,t\}$ is a tail event, and so it follows from the Kolmogorov zero-one law that it has full probability.
\end{proof}

\end{document}